\documentclass[journal,twoside,web]{ieeecolor}

\usepackage{fontspec}

\usepackage{generic}
\usepackage{cite}
\usepackage{amsmath,amssymb,amsfonts}
\usepackage{graphicx}
\usepackage{algorithm,algorithmic}
\usepackage{hyperref}
\hypersetup{hidelinks=true}
\usepackage{textcomp}
\usepackage[utf8]{inputenc}
\usepackage{changepage}
\usepackage{graphics,graphicx}
\usepackage{mathrsfs} 
\usepackage{comment}
\usepackage{algorithm}
\usepackage{algorithmic}
\usepackage[linesnumbered,ruled,vlined,algo2e]{algorithm2e}
\usepackage{xcolor}
\usepackage{comment}
\usepackage{bm}
\usepackage{bbm}
\usepackage{enumerate}
\usepackage{commath}
\usepackage{subfigure}
\usepackage{cite}
\usepackage{booktabs}
\usepackage{empheq}
\usepackage{amssymb}

\usepackage{amsthm}
\usepackage{accents}
\usepackage{thmtools}
\usepackage{thm-restate}
\usepackage{float}
\usepackage[normalem]{ulem}
\usepackage{tikz-cd}
\usepackage{arydshln}

\theoremstyle{plain}
\newtheorem{corollary}{Corollary}

\newtheorem{definition}{Definition}
\newtheorem{theorem}{Theorem}

\newtheorem{proposition}{Proposition}
\newtheorem{problem}{Problem}
\newtheorem*{problem*}{Problem}
\newtheorem*{theorem*}{Theorem}
\newtheorem{assumption*}{Assumption}

\newtheorem{remark}{Remark}

\theoremstyle{definition}
\newtheorem{example}{Example}

\def\BibTeX{{\rm B\kern-.05em{\sc i\kern-.025em b}\kern-.08em
    T\kern-.1667em\lower.7ex\hbox{E}\kern-.125emX}}

\begin{document}
\title{Differential Geometric Conditions for Koopman Linearizability of Control-Affine Systems}
\author{Shankar A. Deka, \IEEEmembership{Member, IEEE}
\thanks{Shankar A. Deka is with Department of Electrical Engineering and Automation, School of Electrical Engineering,
        Aalto University, Maarintie 8, 02150 Espoo, Finland.
        {email: \tt\small shankar.deka@aalto.fi}}}

\maketitle

\begin{abstract}
Koopman linearization opens many possibilities for control synthesis and analysis of nonlinear systems. Whether or not any given nonlinear control system admits a finite-dimensional Koopman representation remains a crucial question to address. A related problem is to categorize the class of all Koopman linearizable nonlinear control systems. In this work, we present differential geometric conditions on the drift and control vector fields of a control-affine nonlinear system, that must be necessarily satisfied for Koopman linear transformation to exist. The same conditions are also shown to be sufficient for (a slightly weaker notion of) Koopman linearizability on control-invariant manifolds. Further, these conditions, together with an additional condition, become necessary and sufficient for Koopman linearizability to a controllable linear system. Our examples illustrate the ease of checking these conditions, and also shed light on how Koopman linearizing transformation may not exist for a control-affine system even though one can linearize the autonomous part of the system via Koopman lifting.  
\end{abstract}

\section{Introduction}
\IEEEPARstart{K}{oopman} representations of nonlinear dynamical systems greatly aid in their study, by allowing one to utilize a wide set of techniques that are otherwise suited only for linear systems. For example, they have been used for forecasting, reachability verification, and stability analysis \cite{azencot2020forecasting}\cite{ding2024time}\cite{deka2022koopman}. In the presence of control inputs, Koopman-based linearization has still been pursued with great success in nonlinear control synthesis problems, most notably in conjunction with linear MPC \cite{korda2018linear}. 

While a plethora of work demonstrates the indisputable utility of the Koopman framework in control and analysis of nonlinear systems, along with an equally prolific body of work on data-driven algorithms to approximate Koopman representations, a comparatively small number of works in the literature have focused on rigorously answering the questions surrounding the existence of finite-dimensional Koopman representations. Starting with autonomous systems, for example, \cite{mezic2021koopman} provides a necessary and a sufficient condition for linearizing injective maps. The existence of linearizing embeddings on compact attractors and invariant sets through an algebraic geometry and topology viewpoint is studied in \cite{kvalheim2026linearizability}. \cite{liu2025properties} investigates the existence of continuous one-to-one linear immersions for nonlinear systems with multiple omega-limit sets. Sufficient conditions for so-called super-linearization based on properties of the graph associated with the system's vector field jacobian, is presented in \cite{belabbas2023sufficient}. 

Extending these types of existence results on Koopman embedding to controlled nonlinear systems is a natural albeit challenging next step, as naive extensions from autonomous to controlled systems are severely restricted \cite{haseli2025two}. Several interesting works have emerged over the recent years towards Koopman representation for control. For instance, \cite{haseli2026modeling} studies discrete-time systems and the general form of Koopman linearization using the idea of common invariant subspaces. Koopman bilinear forms are considered in \cite{goswami2021bilinearization} and more recently, \cite{katayama2026global} provides sufficient conditions for their existence based on existence of an isomorphism between two Lie algebras. Another interesting extension of Koopman representation to controlled system was presented in \cite{iacob2024koopman}, wherein linear parameter varying forms are leveraged to handle the state-dependence of the input matrix. The authors in \cite{ko2024minimum} note that lifting to linear systems, whenever possible, can be performed in infinitely many ways, and thus focus their research on identifying the minimal number of observables needed to do so. Most recently, \cite{shang2026existence} has identified structural requirements (up to a similarity transformation) for discrete-time control systems along with an additional condition on Koopman linearizability of its autonomous subsystem, for the existence of Koopman embedding. In a complementary direction, considering the difficulty in finding an exact Koopman transformation that renders a controlled system linear, works such as \cite{strasser2026overview} have instead stressed on the importance of quantifying approximation errors arising from data-driven algorithms for Koopman linearization.

\subsection{Contributions}
Towards addressing the fundamental question surrounding the existence of Koopman linear embeddings for nonlinear control-affine systems, this work builds upon a rich history of geometric control theory \cite{Nijmeijer1990NonlinearSystems}, where we adopt/extend related results in differential geometry that enable us to tackle specific challenges that arise within the context of Koopman linearization. Our contributions are as follows:
\begin{itemize}
    \item We first present necessary conditions for the existence of diffeomorphic Koopman transformations to a linear control system, along with the domain over which such a transformation is valid. These results are directly based on the drift and the control vector fields, and take the form of Lie algebraic rank condition and commutativity condition. 
    \item We show that these conditions are also sufficient for the existence of linearizing diffeomorphic maps between control invariant sets. When transformations are limited to controllable linear systems, this restriction on the transformation domain is completely removed. That is, we get a stronger, ``if and only if" condition for Koopman linearizability.
    \item Consequently, our results lead to the categorization of the class of control-affine nonlinear systems that can and cannot be Koopman linearized, thereby imposing fundamental restrictions on learning algorithms that seek to perform Koopman linearization of controlled systems.
    \item We draw direct connections between our conditions and those for feedback linearizability, wherein, for a special case of a single-input system, satisfaction of the former trivially implies satisfaction of the latter. We present illuminating analytical examples that demonstrate that our results are readily applicable and straightforward to check, as they do not depend on the existence of additional abstract constructs.
\end{itemize} 
 
\subsection{Related work} Seeking coordinate transformations that linearize nonlinear dynamical systems dates back to the 1970s, starting with the work of Krener \cite{Krener1973}, who presented the equivalence of two nonlinear systems in terms of the existence of a linear map between two linear spaces constructed using Lie brackets of the system's drift and control vectors. A few years later, Brockett \cite{Brockett1978FeedbackSystems} proposed an additional transformation to the control input to linearize single-input systems into controllable linear systems (which is now widely known as feedback linearization). Additional considerations on the linearization domain were explored in \cite{Hunt1983GlobalSystems, Boothby1984SomeSystems}. By 1985, further extensions to multi-input systems were made in \cite{Respondek1985GeometricSystems, Cheng1985GlobalFeedback}. Another practical addition to this body of work was made by \cite{Cheng1988ExactOutputs} wherein linearization of the output was also considered together with the dynamics.

These historic works \cite{Krener1973}-\cite{Cheng1988ExactOutputs}, however, are fundamentally different from modern-day Koopman-based linearization in two ways. Firstly, they differ in terms of the dimensions of the original and transformed state space. In most cases, Koopman linearization allows embedding the original state space into a higher-dimensional space, which immediately breaks the controllability assumption. Controllability of the transformed linear system is key in static/feedback-linearization literature, but is not as important in Koopman linearization. A second, more obvious difference is that Koopman linearization leads to the control inputs being unchanged in the original and lifted system, i.e., only the states are transformed. This distinction can make the control synthesis and implementation easier, since the same input drives both the nonlinear and the transformed linear system. Moreover, foregoing input transformation allows one to restrict the search for the linearizing coordinate transformation (\textit{if it exists}) using eigenfunctions of the autonomous part of the nonlinear dynamics. The existence of such a linearizing transformation itself is the main question that we seek to answer in this paper. 

The contributions of our paper are organized into the following sections. We start with a background on Koopman Operator theory, essential definitions, problem formulation, and key preliminary results in Section \ref{sec:prel}. The main technical results are presented in Section \ref{sec:main}, where we derive the necessary conditions and sufficient conditions for Koopman linearizability of a nonlinear system. Section \ref{sec:controllab} discusses conditions for the preservation of controllability under Koopman transformations. Our technical results are illustrated with examples in Section \ref{sec:results}. Finally, we present some conclusions and discuss future directions in Section \ref{sec:conc}.\\

\textbf{Notation.} The space of k-times continuously differentiable functions from a set $X$ to $Y$ is given by $\mathcal{C}^k(X,Y)$. The jacobian of a (vector-valued) function $f$ w.r.t. variable ${*}$ is denoted by $\nabla_{*}f$. Given two vector fields $v,w:X\subset \mathbbm{R}^n \rightarrow \mathbbm{R}^n,$ the Lie bracket is denoted as $[v,w](x) = \nabla_x v(x)\cdot w(x) - \nabla_x w(x)\cdot v(x)$. The vector $\mathbbm{1}_i \in \mathbbm{R}^n$ denotes the standard basis vector with $i^{th}$ component equal to one.

\section{Preliminaries}\label{sec:prel}
We consider a continuous-time nonlinear control-affine system
\begin{equation}\label{eq:nonlinear}
\dot{x}=f(x)+\sum_{i=1}^m g_i(x) u_i
\end{equation}
where states $x\in X \subseteq \mathbbm{R}^n$ and $f, g_1, \ldots, g_m$ are smooth vector fields on $X$. The control input vector $u \doteq [u_1,\ldots,u_m]^\top \in 
\mathbbm{R}^m.$

\begin{definition}[Koopman linearizability]
    We shall refer to a control-affine nonlinear system \eqref{eq:nonlinear} as \textit{Koopman linearizable} if there exists a set $X_0  \subseteq X$, a diffeomorphism $\Phi \in \mathcal{C}^2(X_0,\mathbbm{R}^N)$ from $X_0$ onto its image with $N\ge n$, and matrices $A$ and $B$, such that 
    the transformed state $z(t) = \Phi(x(t))$ evolves according to a linear dynamics 
    \[
     \frac{d}{dt}z(t) = Az(t) + Bu(t),
    \]
    for any input $u(\cdot)$, initial state $x(0)\in X_0$, and  $t<\inf\left\{\tau >0 \mid x(\tau) \notin X_0\right\}$.
\end{definition}
With this definition of Koopman linearizable nonlinear systems, we are interested in answering the following question:

\begin{problem}
Given a nonlinear control-affine system of the form \eqref{eq:nonlinear}, find conditions on the drift and the control vector fields ($f(x)$ and $g_i(x)$) for which the system is Koopman linearizable.
\end{problem}

It is important to note here, that our problem involves finding conditions that can be checked \textit{directly}, using only the provided drift and control vector fields. Typically, Koopman linearizability is commonly assumed in `applied Koopmanism' literature and relies on the existence of a set of observables $\psi_1,\ldots,\psi_N$ such that 
\begin{enumerate}
    \item[(a)] span$\{\psi_i \mid i=1,\ldots,N\}$ is Koopman invariant with respect to the autonomous dynamics $\dot{x} = f(x)$, and
    \item[(b)] Each $g_i(x)$ is spanned by the columns of the pseudo-inverse of the Jacobian matrix $\nabla_x \Psi$, where $\Psi(x) = [\Psi_1(x),\ldots,\Psi_N(x)]^\top$.
\end{enumerate}
Clearly, these two conditions are indirect and not straightforward to check for a given dynamics \eqref{eq:nonlinear}. In fact, approaching the problem of Koopman linearizability for controlled systems by first finding a Koopman invariant subspace for the autonomous dynamics (by setting $u=0$) can be problematic since two different systems with the same autonomous dynamics can yield different outcomes (as we show in our Example \ref{ex:slow_manifold1} and Example \ref{ex:slow_manifold2}). Instead, we seek conditions that only rely on the vector fields $f$ and $g_i$'s to draw conclusions on Koopman linearizability. 

A key distinction between this notion of linearization and the more broadly known feedback linearization is that the former retains the same control inputs as the original nonlinear dynamics. Nevertheless, the same differential geometric tools (e.g. Lie brackets and algebra, rank of distributions, involutions, etc.) can be adapted to analyze Koopman linearizability of nonlinear systems. Towards that end, we present some essential notations, definitions, and preliminary results. In order to make the materials accessible to a broader controls community, we depart from some standard notions commonly found in mathematical textbooks in differential geometry (where, e.g., vector fields are often represented as linear maps from the space of smooth functions onto itself \cite{lee2003smooth}\cite{olver1993applications}), and develop our results in a more familiar notational setting.

\begin{definition}[Adjoint of a vector field]
For any two vector fields $f$ and $g$, we define repeated Lie bracket in terms of adjoint representation, $\operatorname{ad}_f^k g$, $k=0,1,2, \ldots$, inductively as $\operatorname{ad}_f^k g=\left[f, \operatorname{ad}_f^{k-1} g\right], k \geq 1$, with $\operatorname{ad}_f^0 g=g$.
\end{definition}

\begin{definition}[Involutive distribution] Given a set of $d$ smooth vector fields $v_1(x),\,v_2(x),\ldots,v_d(x)$ on an open set $U \subseteq \mathbbm{R}^n$, the map
\[
    \Delta(x) = \operatorname{span}\left\{v_1(x),\,v_2(x),\, \ldots, v_d(x) \right\}
\]
from any point $x\in U$ to a vector space is called a distribution. The distribution $\Delta(x)$ is called an involutive distribution if $[v_i(x),v_j(x)]\in \Delta(x)$ for each $1\le i,j \le d.$
\end{definition}
\begin{proposition}[Feedback linearization \cite{isadoriNonlinear}]\label{prop:FL}
Consider the control-affine nonlinear system \eqref{eq:nonlinear} with a single input (i.e. $m=1$) as follows:
    \begin{equation}\label{eq:single-input}
        \dot{x}=f(x) + g_1(x) u_1.
    \end{equation}
This system is feedback linearizable to a controllable system if and only if the following conditions hold:
\begin{enumerate}
    \item[(i)] The matrix $\left\{g_1(x_0) \mid \operatorname{ad}_f g_1(x_0) \mid \ldots \mid \operatorname{ad}_f^{n-1} g_1(x_0) \right\}$ has rank $n$, and
    \item[(ii)] $\operatorname{span}\left\{g_1(x),\,\operatorname{ad}_f g_1(x),\, \ldots, \operatorname{ad}_f^{n-2} g_1(x) \right\}$ is an involutive distribution in a neighborhood of $x_0$.
\end{enumerate}
\end{proposition}

We will return to this proposition later in Section \ref{sec:controllab} when we compare the necessary and sufficient conditions for Controllable Koopman linearizability to feedback linearization. Next, we present another key result in differential geometry, which is often referred to as the straightening theorem or flow-box theorem, generalized to $n-$dimensions, and is closely related to the well-known Frobenius theorem \cite{Nijmeijer1990NonlinearSystems}\cite{lee2003smooth}\cite{olver1993applications}.

\begin{proposition}[Flow-box theorem generalization]\label{prop:box}
    Consider a set of $n'\le n$ independent vector fields $v_1(x),\,v_2(x),\ldots,v_m(x)$ on an open set $U \subseteq \mathbbm{R}^n$, and let $[v_i(x),v_j(x)]=0$ for all $0\le i,j \le n'$. Then, there exists a coordinate transformation $\Phi: U \rightarrow \mathbbm{R}^n$ such that
    \begin{equation}\label{eq:rectify}
        \nabla_x\Phi(x)\cdot v_i(x) = \mathbbm{1}_i
    \end{equation}
for all $x\in U.$
\end{proposition}
Figure \ref{fig:frobenius} illustrates how this transformation $\Phi$ can be applied to rectify nonlinear vector fields. Note that if a transformation $\Phi$ rectifies the vector fields $v_i(x)$ (i.e., satisfies equation \eqref{eq:rectify}), then any transformation $\tilde{\Phi}(x) \doteq \Phi(x)-\Phi(x')$ for any arbitrary point $x'\in X$, also rectifies the vector fields $v_i(x)$. This means $\tilde{\Phi}(x') = 0.$ Thus, without loss of generality, we can set the root of the rectifying transformation $\Phi$ to be equal to any point $x'$ as per our convenience.  We are now ready to prove the main results in the next section.

\begin{figure}[h]
    \centering
    \includegraphics[width=0.9\linewidth, trim=0 20 0 5, clip]{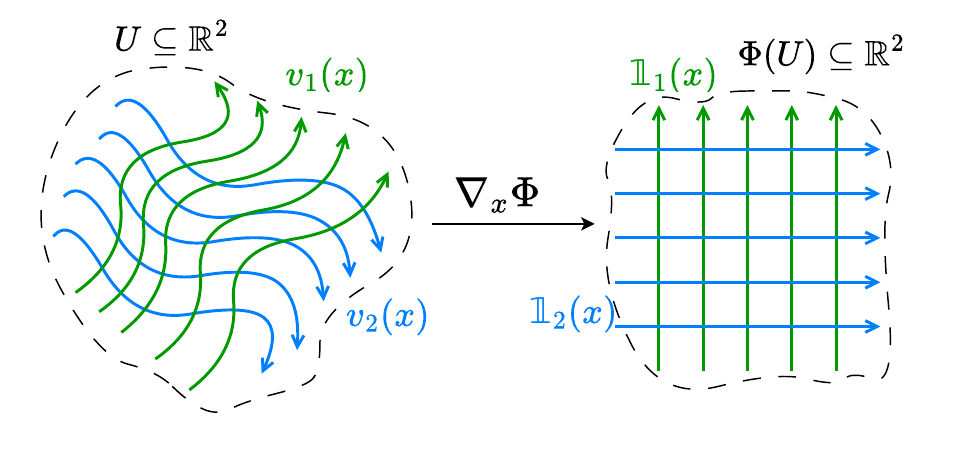}
    \caption{As per Proposition \ref{prop:box}, for independent, commuting 2-d vector fields $v_1(x)$ and $v_2(x)$ (i.e., $[v_1(x), v_2(x)] = 0$), the map $\nabla_x\Phi$ ``straightens" the vector fields on the domain $U.$ The green lines correspond to the integral curves of the vector field $v_1(x)$ and $\mathbbm{1}_1(x)$, while the blue lines illustrate the integral curves for $v_2(x)$ and $\mathbbm{1}_2(x).$}
    \label{fig:frobenius}
\end{figure}

\section{Main results}\label{sec:main}
We first show that diffeomorphic transformations, including lifting transformations, preserve Lie brackets between vector fields. This is stated more concretely as follows.

\begin{proposition}[Coordinate invariance]\label{prop:lie_invariance}
    Let $\Phi \in \mathcal{C}^2(X,Z)$ be a diffeomorphism from $X$ onto its image $Z$. Then for any vector fields $v_1(x),v_2(x) \in \mathcal{C}^1(X,\mathbbm{R}^n)$, let $\tilde{v}_1(z) \doteq \big(\nabla_x \Phi(x) \cdot v_1(x)\big)\lvert_{x=\Phi^{-1}(z)} $ and $\tilde{v}_2(z) \doteq \big(\nabla_x \Phi(x) \cdot v_2(x)\big)\lvert_{x=\Phi^{-1}(z)}$ be two vector fields on $Z$. Then, we have
    \begin{equation}
        \nabla_x\Phi(x) \cdot [v_1(x),\,v_2(x)] = [\tilde{v}_1(z),\, \tilde{v}_2(z)].
    \end{equation}
\end{proposition}
\begin{proof}
    We start by expanding the Lie bracket $[\tilde{v}_1(z),\,\tilde{v}_2(z)] =  \nabla_z \tilde{v}_1(z) \cdot \tilde{v}_2(z) - \nabla_z \tilde{v}_2(z)\cdot \tilde{v}_1(z)$. We note that by the chain rule, $\nabla_z \tilde{v}_1(z) = \nabla_x(\nabla_x \Phi(x) \cdot v_1(x)) \cdot \nabla_z \Phi^{-1}(z)$. Further, we note that for any $x\in X$
    \begin{gather*}
        \Phi^{-1}\left(\Phi(x)\right) = x \implies \\
         \nabla_x \left[ \Phi^{-1}\left(\Phi(x)\right) \right] = \left[\nabla_z\Phi^{-1}(z)\right]_{z=\Phi^{-1}(x)} \cdot \nabla_x\Phi(x) = I_{n\times n}
    \end{gather*}
    Thus, 
    \begin{eqnarray*}
      &&\nabla_z \tilde{v}_1(z)\cdot \tilde{v}_2(z)\\
      &=& \nabla_x(\nabla_x \Phi(x) \cdot v_1(x))\cdot \underbrace{\nabla_z\Phi^{-1}(z) \cdot \nabla_x\Phi(x)}_{= I_{n\times n}} \cdot v_2(x)  \\ 
      &=& \nabla_x(\nabla_x \Phi(x) \cdot v_1(x)) v_2(x).
    \end{eqnarray*} 
    Now, by the product rule, this is equal to 
\begin{equation*}
    \begin{split}
        &\Big(\nabla_x \Phi(x) \cdot \nabla_x v_1(x)  + \sum_i v_1^i(x)\frac{\partial}{\partial x_i}\nabla_x\Phi(x)\Big)\cdot v_2(x)\\ 
        = \quad &\nabla_x \Phi(x) \cdot \nabla_x v_1(x)\cdot v_2(x) \quad + \\
        &\quad\quad\Big(\sum_i v_1^i(x)\frac{\partial}{\partial x_i}\nabla_x\Phi(x)\Big)\cdot v_2(x)\\
        = \quad &\nabla_x \Phi(x) \cdot \nabla_x v_1(x)\cdot v_2(x) \quad +\\
        &\quad\quad\Big(\sum_i v_1^i(x) \sum_j v_2^j(x)\frac{\partial^2 \Phi(x)}{\partial x_i \partial x_j}\Big)\\
         = \quad &\nabla_x \Phi(x) \cdot \nabla_x v_1(x)\cdot v_2(x) \quad +\\
         & \quad\quad\Big(\sum_i \sum_j v_1^i(x)  v_2^j(x)\frac{\partial^2 \Phi(x)}{\partial x_i \partial x_j}\Big)
    \end{split}  
\end{equation*}
In the same way, we get
\begin{align*}
    \nabla_z \tilde{v}_2(z)\cdot \tilde{v}_1(z) &= \nabla_x \Phi(x) \cdot \nabla_x v_2(x)\cdot v_1(x)\\
    &+ \Big(\sum_i \sum_j v_2^i(x)  v_1^j(x)\frac{\partial^2 \Phi(x)}{\partial x_i \partial x_j}\Big).
\end{align*}
    Note that the second term of $\nabla_z \tilde{v}_1(z)\cdot \tilde{v}_2(z)$ and $\nabla_z \tilde{v}_2(z)\cdot \tilde{v}_1(z)$ are the same. Thus, by subtracting $\nabla_z \tilde{v}_1(z)\cdot \tilde{v}_2(z)$ and $\nabla_z \tilde{v}_2(z)\cdot \tilde{v}_1(z)$, we finally have
    \begin{align*}
        [\tilde{v}_1(z),\,\tilde{v}_2(z)] &=  \nabla_x \Phi(x) \cdot \nabla_x v_1(x)\cdot v_2(x)\\
        &- \nabla_x \Phi(x) \cdot \nabla_x v_2(x)\cdot v_1(x)\\
        &= \nabla_x\Phi(x)\cdot [v_1(x), v_2(x)].
    \end{align*}
    This completes our proof. \qed
\end{proof}

We point out that the proposition \ref{prop:lie_invariance} can in general, be applied to coordinate transformations that involve lifting to a higher-dimensional space (i.e., $Z$ can have a dimension $N$ greater than $n$). This key extension allows us to utilize existing differential geometric machinery more broadly towards Koopman-based linearization, which often involves lifting. The commutative diagram shown in Figure \ref{fig:commutative-diagram-Phi} summarizes this idea.


\begin{figure}[t]
\begin{tikzcd}[column sep=2.2em]
    |[draw, thick, rounded corners, fill=blue!5]|{\hspace{-0.2em}\mathcal{C}^1(X,\mathbb{R}^n)\hspace{-0.2em}}
    && |[draw, thick, rounded corners, fill=green!5]|{\hspace{-0.2em}\mathcal{C}^1(X,\mathbb{R}^N)\hspace{-0.2em}}
    && |[draw, thick, rounded corners, fill=red!5]|{\hspace{-0.2em}\mathcal{C}^1(Z,\mathbb{R}^N)\hspace{-0.2em}} \\
    \\
    |[draw, thick, rounded corners, fill=blue!5]|{\hspace{-0.2em}\mathcal{C}^1(X,\mathbb{R}^n)\hspace{-0.2em}}
    && |[draw, thick, rounded corners, fill=green!5]|{\hspace{-0.2em}\mathcal{C}^1(X,\mathbb{R}^N)\hspace{-0.2em}}
    && |[draw, thick, rounded corners, fill=red!5]|{\hspace{-0.2em}\mathcal{C}^1(Z,\mathbb{R}^N)\hspace{-0.2em}}
   	\arrow["{\nabla_x\Phi}", from=1-1, to=1-3]
	\arrow["{[\cdot\,,\,\cdot](x)}"', from=1-1, to=3-1]
	\arrow["{x=\Phi^{-1}(z)}", from=1-3, to=1-5]
	\arrow["{[\cdot\, ,\, \cdot](z)}", from=1-5, to=3-5]
	\arrow["{\nabla_x\Phi}"', from=3-1, to=3-3]
	\arrow["{x=\Phi^{-1}(z)}"', from=3-3, to=3-5]
\end{tikzcd}
    \caption{Commutative diagram describing the transformation of $\mathcal{C}^1$-functions and their brackets under the Koopman lifting $z = \Phi(x)$. Under Proposition \ref{prop:lie_invariance}, the Lie brackets are conserved under coordinate change, i.e., the transformed Lie bracket between two vector fields is the same as the Lie bracket of the transformed vector fields.}
    \label{fig:commutative-diagram-Phi}
\end{figure}
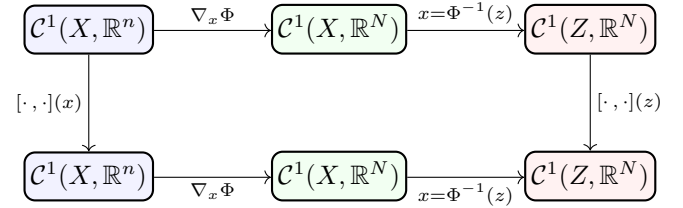

\begin{example}
    Consider two vector fields on $\mathbbm{R}^2$, given by $v_1(x) = \left[\begin{matrix}x_1\\x_2^2 \end{matrix}\right]$ and $v_2(x) = \left[\begin{matrix}x_1^2\\x_2 \end{matrix}\right]$ , with $[v_1,v_2]=\left[\begin{matrix}x_{1}^{2}\\x_{1} \left(1 - 2 x_{2}\right)\end{matrix}\right]$ and a diffeomorphic lifting map $\Phi$:$$\left[\begin{matrix}x_1\\x_2 \end{matrix}\right] \overset{\Phi} \mapsto \left[\begin{matrix}z_1\\z_2\\z_3\end{matrix}\right] = \left[\begin{matrix}\tanh{\left(x_{1} + x_{2} \right)}\\x_{1}^{2}\\\sin{\left(x_{1} \right)}\end{matrix}\right],$$
with an inverse 
$$
\left[\begin{matrix}z_1\\z_2\\z_3\end{matrix}\right] \overset{\Phi^{-1}} \mapsto \left[\begin{matrix}x_1\\x_2 \end{matrix}\right] = \left[\begin{matrix}\frac{z_3\sqrt{z_2}}{\sin{(\sqrt{z_2})}} \\ \operatorname{atanh}\left({z_1 - \frac{z_3\sqrt{z_2}}{\sin{(\sqrt{z_2})}}}\right) \end{matrix}\right].
$$
The Lie bracket transformation $\nabla_x\Phi \cdot [v_1,\,v_2]$ is equal to
\begin{equation}\label{eq:ex1}
\left[\begin{matrix} \left(x_{1} + x_{1}^2 - 2 x_1 x_{2}\right) \left(1 - \tanh^{2}{\left(x_{1} + x_{2} \right)}\right)\\2 x_{1}^{3}\\x_{1}^{2} \cos{\left(x_{1} \right)}\end{matrix}\right].
\end{equation}
Now, in $\left\{(x_1,x_2)\mid x_2\ge 0\right\}$, the transformed vector fields in lifted $z-$coordinates are given by 
\[
    \tilde{v}_1(z) = \left[\begin{matrix} \left(1 - z_{1}^{2}\right) \left(\sqrt{z_{2}} +  \left(- \sqrt{z_{2}} + \operatorname{atanh}{\left(z_{1} \right)}\right)^{2}\right)\\2 z_{2}\\\sqrt{z_{2}} \cos{\left(\sqrt{z_{2}} \right)}\end{matrix}\right]
\] and
\[
    \tilde{v}_2(z) = \left[\begin{matrix}\left(1 - z_{1}^{2}\right)\left(\sqrt{z_{2}} + z_2\right)  \\2 z_{2}^{\frac{3}{2}}\\z_{2} \cos{\left(\sqrt{z_{2}} \right)}\end{matrix}\right].
\]
Thus, the Lie bracket $[\tilde{v}_1(z),\, \tilde{v}_2(z)]$ is equal to 
\begin{gather*}\small
\left[\begin{matrix} (\sqrt{z_2} - z_1^2\sqrt{z_2})(1 - 2\operatorname{atanh}{\left(z_{1} \right)}) +  3 z_{2} - 3z_1^2 z_2\\2 z_{2}^{\frac{3}{2}}\\z_{2} \cos{\left(\sqrt{z_{2}} \right)}\end{matrix}\right]\\
= \left[\begin{matrix} \left(\sqrt{z_2}(1 - 2\operatorname{atanh}{\left(z_{1} \right)}) +  3 z_{2}\right) (1 - z_1^2)\\  2 z_{2}^{\frac{3}{2}}  \\ z_{2} \cos{\left(\sqrt{z_{2}} \right)}\end{matrix}\right]
\end{gather*}
Substituting $z_1=\tanh{\left(x_{1} + x_{2} \right)}$ and $z_2 = x_1^2$ in the above leads to the expression in equation \eqref{eq:ex1}. The same can be verified for the set $\left\{(x_1,x_2)\mid x_2\le 0\right\}$.
\end{example}

\begin{theorem} [Necessity for Koopman linearization]\label{th:Neccesary}
Consider the nonlinear system \eqref{eq:nonlinear} with $f\left(x_0\right)=0$. If the system \eqref{eq:nonlinear} is Koopman linearizable, then the following two conditions hold:
\begin{enumerate}
    \item[(i)] $\operatorname{dim}\left(\operatorname{span}\left\{\operatorname{ad}_f^j g_i(x), 1\le i \le m,\, j\ge 0\right\}\right)$ is constant for all $x \in \tilde{V}$, and
    \item[(ii)] $\left[\operatorname{ad}_f^k g_i, \operatorname{ad}_f^l g_j\right](x)=0, \quad$ for all $1\le i, j \le m$ and $  k, l \ge 0, \quad \forall x \in \tilde{V}$,
\end{enumerate}
where $\tilde{V}$ is an open neighborhood of $x_0.$
\end{theorem}

\begin{proof}
   Suppose $\Phi: X \to Z$
    is a (diffeomorphic) Koopman transformation, such that the transformed state $z=\Phi(x)$ has linear dynamics 
    \begin{equation}\label{eq:th_lin_dyn}
        \dot{z} = \nabla_x \Phi(x) \cdot(f(x)+\sum_{i=1}^m g_i(x) u_i) = Az + B u.
    \end{equation}
    In other words, $\nabla_x \Phi(x)\cdot f(x) = Az$ and $\nabla_x \Phi(x) \cdot \operatorname{ad}_f^0 g_i(x) = \nabla_x \Phi(x)\cdot g_i(x) = b_i$, where $b_i$ is the $i^{th}$ column of the matrix $B.$ By proposition \ref{prop:lie_invariance}, we have $ \nabla_x \Phi(x) \cdot \operatorname{ad}_f g_i(x) = \nabla_x \Phi(x) \cdot[f(x),\,g_i(x)] = [Az,\,b_i] = Ab_i.$ Utilizing proposition \ref{prop:lie_invariance} again yields
     $ \nabla_x \Phi(x) \cdot \operatorname{ad}_f^2 g_i(x) = \nabla_x \Phi(x) \cdot[f(x),\, \operatorname{ad}_f g_i(x)] = [Az,\,Ab_i] =  A^2b_i,$ and more generally, we have 
     \begin{eqnarray*}
         \nabla_x \Phi(x) \cdot \operatorname{ad}_f^k g_i(x) &=& \nabla_x \Phi(x) \cdot[f(x),\, \operatorname{ad}_f^{k-1} g_i(x)]\\
         &=&  A^kb_i
     \end{eqnarray*}
     for all non-negative integers $k$. Now, given that $\nabla_x \Phi(x) \cdot \operatorname{ad}_f^k g_i(x) = A^kb_i$ and $\nabla_x \Phi(x) \cdot \operatorname{ad}_f^l g_j(x) = A^lb_j$, applying proposition \ref{prop:lie_invariance} one final time gives $\nabla_x \Phi(x) \cdot [\operatorname{ad}_f^k g_i(x),\, \operatorname{ad}_f^l g_j(x)] = [A^kb_i,\, A^lb_j] = 0,$ for all non-negative integers $k,l$ and for all $i,j \in \{1,\ldots,m\}.$ This means $[\operatorname{ad}_f^k g_i(x),\, \operatorname{ad}_f^l g_j(x)] = 0,$ for all non-negative integers $k,l$ and for all $i,j \in \{1,\ldots,m\}$ since $\nabla_x\Phi(x)$ is full column rank on $\tilde{V}$.
     
     Next, to show condition (i) must hold, we again note that $ \nabla_x \Phi(x) \cdot \operatorname{ad}_f^j g_i(x) =  A^jb_i$ on the set $\tilde{V}.$ Thus, \begin{eqnarray}\label{eq:rank}
         &&\operatorname{span}\left\{A^{j}b_i \mid 1\le i \le m,\, j\ge 0 \right\} \\
         = &&\operatorname{span}\left\{\nabla_x \Phi(x) \cdot\operatorname{ad}_f^j g_i(x) \mid 1\le i \le m,\, j\ge 0 \right\} \nonumber \\
         = &&\nabla_x \Phi(x) \cdot \operatorname{span}\left\{\operatorname{ad}_f^k g_i(x) \mid 1\le i \le m,\, j\ge 0 \right\}. \nonumber
     \end{eqnarray} 
     Now, since $\Phi$ is a diffeomorphism, $\nabla_x\Phi$ is non-singular on $\tilde{V}$ and thus, $ \operatorname{dim}\left(\operatorname{span}\left\{\operatorname{ad}_f^j g_i(x)\right\}\right) = \operatorname{dim}\left(\nabla_x \Phi(x) \cdot \operatorname{span}\left\{\operatorname{ad}_f^j g_i(x)\right\}\right)$ for $j\ge 0$ and $i=1,\ldots,m$, which is a constant no greater than $n$, due to equation \eqref{eq:rank}. This concludes our proof.
    \qed
\end{proof}

The same conditions that appear in Theorem \ref{th:Neccesary}, can also be used to establish a special case of Koopman linearization (in the sense that one can show the existence of a Koopman transformation that projects the original states to a lower-dimensional manifold). This slightly weakened result is mainly due to a technicality in the flow-box theorem (Proposition \ref{prop:box}), wherein the coordinate transformation $\Phi$ exists between two spaces of the \textit{same dimension}.

\begin{figure*}[t]
    \centering
\includegraphics[width=0.9\textwidth,trim=5 5 5 5,clip]{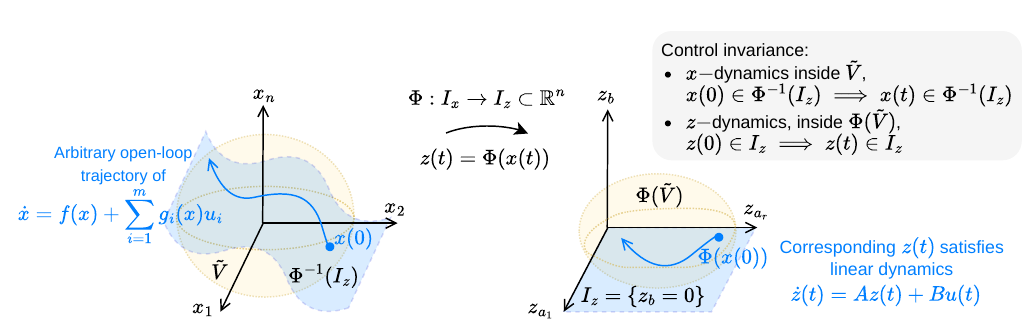}
    \caption{Illustration of Theorem \ref{th:Linear_lower}, which presents sufficient conditions for Koopman linearizability on a manifold $I_x\doteq \Phi^{-1}(I_z)$, shown in blue color, that is control invariant inside an open set $\tilde{V}$, shown in buff color. In transformed coordinates, this manifold is mapped to the $I_z$ plane, which is control invariant inside the open set $\Phi(\tilde{V})$. The dimension of the manifolds $I_x$ and $I_z$ is equal to the dimension of the distribution in condition (i) of the theorem.}
    \label{fig:Theorem2}
\end{figure*}

\begin{theorem}[Sufficiency for Koopman linearization]\label{th:Linear_lower}
Consider again the nonlinear system \eqref{eq:nonlinear} with $f\left(x_0\right)=0$. Suppose that in a neighborhood $\tilde{V}$ of $x_0$, we have
\begin{enumerate}
    \item[(i)] $\operatorname{dim}\left(\operatorname{span}\left\{\operatorname{ad}_f^j g_i(x), 1\le i \le m,\, j\ge 0 \right\}\right)$ is constant for all $x \in \tilde{V}$, and
    \item[(ii)] $\left[\operatorname{ad}_f^k g_i, \operatorname{ad}_f^l g_j\right](x) = 0,$ for all $1\le i, j \le m$ and $k, l \ge 0, \quad \forall x \in \tilde{V}$.
\end{enumerate}
Then, there exists a manifold $I_x \subset X$ and a coordinate transformation $\Phi: I_x \rightarrow \mathbbm{R}^n$ such that the system \eqref{eq:nonlinear} is Koopman linearizable.  
\end{theorem}
\begin{proof}
    Since the distribution $M \doteq \operatorname{span}\left\{\operatorname{ad}_f^j g_i(x), \,1\le i \le m,\, j\ge 0 \right\}$ has some constant dimension, $r$, for all $x \in \tilde{V}$, we can pick $r$ independent vectors $v_1(x),\ldots,v_r(x) \in M$ such that $[v_i,v_j] = 0.$ Thus, due to Proposition \ref{prop:box}, there exists a diffeomorphism $\Phi(x)$ defined on $\tilde{V}$ such that $\nabla_x\Phi(x)\cdot v_i(x) = \mathbbm{1}_i$ for all $i=1,\cdots,r$. Without loss of generality, we can take $\Phi(x_0)=0$, since proposition \ref{prop:box} 
    pertains only to the gradient of $\Phi$. For notational convenience, we partition the transformed coordinates as 
    $$\Phi(x) = z = \left[ \begin{array}{c}
         z_a  \\
         \hdashline[4pt/2pt]
         z_b 
    \end{array}\right], \text{ with } z_a\in \mathbbm{R}^r, \,z_b\in \mathbbm{R}^{n-r}. $$
 Note that for any $w(x) \in M$, we have $w(x) = \sum_{i=1}^r \alpha_i(x)v_i(x)$ for some scalar functions $\alpha_i.$ This means $\nabla_x\Phi(x)\cdot w(x) = [\alpha_1(x), \ldots,\alpha_r(x), 0, \ldots, 0]^\top$ which means $\tilde{w}(z)\doteq \nabla_x\Phi(x)\cdot w(x)\lvert_{x=\Phi^{-1}(z)}$ has its first $r$ elements non-zero and the remaining $n-r$ elements equal to zero. Additionally, since $\nabla_x\Phi(x)\cdot[w(x),v_i(x)] = 0$ for each $i=1,\ldots,r$, we have $[\tilde{w}(z), e_i] = \frac{\partial \tilde{w}}{\partial z_i} =0$. Thus, $\tilde{w}(z)$ is of the form $\left[\begin{array}{c}
         \tilde{w}_a(z_b)  \\
         0_{(n-r)\times 1} 
    \end{array}\right]$. This in particular means that $\tilde{g}_i(z)$ and $[\tilde{f}(z),\mathbbm{1}_i]$ have the last $(n-r)$ rows zero, since $g_i(x)$ and $[f(x),v_i(x)]$ belong to the set $M.$ We can then write $\tilde{g}_i(z)$ as
    \[
        \tilde{g}_i(z) = \left[\begin{array}{c}
             \tilde{g}_{ai}(z_b)  \\
             0_{(n-r)\times 1} 
        \end{array}\right],
    \]
    for some function $\tilde{g}_{ai}:\mathbbm{R}^{n-r}\rightarrow\mathbbm{R}^{r}.$
    Now, since $[[f(x),v_i(x)],v_j(x)]=0$, we have $[[\tilde{f}(z),\mathbbm{1}_i],\mathbbm{1}_j]=0$ which leads to $\frac{\partial^2 \tilde{f}(z)}{\partial z_i \partial z_j} =0$ for $i,j=1,\ldots,r.$ Thus, for some matrix $A$ and functions $\tilde{f}_a:\mathbbm{R}^{n-r}\rightarrow\mathbbm{R}^r,\,\tilde{f}_b:\mathbbm{R}^{n-r}\rightarrow\mathbbm{R}^{n-r}$, we have 
    \[
        \tilde{f}(z) = \left[ \begin{array}{cc}
            A_{11} & A_{12} \\
            A_{21} & A_{22}
        \end{array}\right]z + \left[\begin{array}{c}
             \tilde{f}_a(z_b)  \\
             \tilde{f}_b(z_b) 
        \end{array}\right]
    \]
    Since $\Phi(x_0)=0$ and $f(x_0)=0,$ we have $\tilde{f}(0)=\nabla_x\Phi(x_0)\cdot f(x_0) = 0$, and thus $\tilde{f}_a(0)=0$ and $\tilde{f}_b(0)=0.$ Next we show that $A_{21}=0.$ As $[f(x),v_i(x)] \in M$ for $i=1,\ldots,r$, the vector $[\tilde{f}(z), \mathbbm{1}_i] = \frac{\partial \tilde{f}(z)}{\partial z_i}$ has the last $(n-r)$ rows equal to zero for $i=1,\ldots,r$. This means the $i^{th}$ column of matrix $A_{21}$ is zero. Since this is true for all $i=1,\ldots,r$, we have $A_{21}=0.$ This finally leads us to
    \begin{eqnarray}\label{eq:z_dyn}
        &\dot{z}& = \nabla_x\Phi(x)\cdot f(x) + \sum_{i=1}^m \nabla_x\Phi(x)\cdot g_i(x)u_i \nonumber\\
        &=& \tilde{f}(z) + \sum_{i=1}^m \tilde{g}_i(x)u_i \\ 
        &=& \left[ \begin{array}{cc} 
            A_{11} & A_{12} \\
            0 & A_{22}
        \end{array}\right]z + \left[\begin{array}{c}
             \tilde{f}_a(z_b)  \\
             \tilde{f}_b(z_b) 
        \end{array}\right] + \sum_{i=1}^m \left[\begin{array}{c}
             \tilde{g}_{ai}(z_b)  \\
             0 
        \end{array}\right]u_i. \nonumber
    \end{eqnarray}
Notice that for the set $I_z = \{ z\,|\, z_b=0 \},$ we have $z_b(t) \equiv 0$ along the dynamics \eqref{eq:z_dyn}, if $z(0)\in I_z.$ Thus, everywhere on the manifold $I_x \doteq \Phi^{-1}(I_z) \cap \tilde{V}$, we can apply the transformation $\Phi: x \mapsto z$ such that
\begin{equation}
    \frac{d}{dt} \Phi(x(t)) = A \Phi(x(t)) + B u(t),   
\end{equation}
with $A = \left[ \begin{array}{cc} 
            A_{11} & A_{12} \\
            0 & A_{22}
        \end{array}\right]$ and $B = \left[\begin{array}{c}
             B_1  \\
             0 
        \end{array}\right]$ with $B1 \doteq 
             \big[\tilde{g}_{a1}(0), \ldots,\tilde{g}_{am}(0)\big]$. 
\qed
\end{proof}

\begin{remark}
Notice that Theorem \ref{th:Linear_lower} says that under the same conditions of Theorem \ref{th:Neccesary}, one can map the original state space $X$ restricted to a manifold $I_x$, to state space $Z$ wherein the dynamics becomes linear. The set $I_x$ is non-empty, since $\Phi(x_0) = 0 \in I_z$ and $x_0 \in \tilde{V}$ which means $x_0 \in I_x$. This linear dynamics is, however, uncontrollable, with $z_b$ as the uncontrollable modes. Figure \ref{fig:Theorem2} summarizes this restricted Koopman linearizability result. We discuss controllability-preserving Koopman transformations in the next section.
\end{remark}

\section{Controllability of Koopman linearization}\label{sec:controllab}
We now start with a controllable nonlinear system, and are interested in the question: under what conditions can the system be Koopman linearized into a controllable linear system? Given a controllable nonlinear control-affine system, let us say that there exists a Koopman lifting transformation $x \mapsto \Phi(x) \doteq z$ from $\mathbb{R}^n$ to $\mathbb{R}^N$ with $N > n$ and $\Phi(0) = 0$, such that we obtain the linear dynamics
\[
\dot{z} = Az + Bu.
\]
Clearly, this linear system in the lifted state space $z$ is not controllable, since for any controller $u(\cdot)$ and initial condition $z(0)=0$, the states $z(t)$ are confined to the $n$-dimensional manifold $Z = \left\{ \Phi(x) \,|\, x \in X \subseteq \mathbb{R}^n \right\}$ embedded in $\mathbb{R}^N$. In other words, not all states in $\mathbb{R}^N$ can be reached from the origin. 

In order to extract the controllable subspace, one can rewrite the system $(A,B)$ using Kalman decomposition into controllable canonical form. We can choose an appropriate coordinate transformation $\tilde{z} = Mz$ for some invertible $N \times N$ matrix $M$ such that
\begin{equation}\renewcommand{\arraystretch}{1.5}
    \dot{\tilde{z}} = \underbrace{
    \begin{bmatrix}\begin{array}{c:c}\tilde{A}_{11} & \tilde{A}_{12} \\ 
\hdashline[4pt/4pt]
0 & \tilde{A}_{22} \end{array}\end{bmatrix}
    }_{= MAM^{-1}}\tilde{z} + \underbrace{
    \begin{bmatrix}\begin{array}{c}\tilde{B}_{1} \\ \hdashline[4pt/2pt] 0 \end{array}\end{bmatrix}
    }_{= MB}u. 
\end{equation}
Additionally, recall this decomposition partitions the states $\tilde{z} = \begin{bmatrix}\begin{array}{c}\tilde{z}_{c} \\ \hdashline[4pt/2pt] \tilde{z}_{u}\end{array}\end{bmatrix}$ into $N_c$ controllable states and $N_u$ uncontrollable states, $\tilde{z}_{c}$ and $\tilde{z}_{u}$, respectively, and the pair $(\tilde{A}_{11}, \tilde{B}_1)$ is controllable. For the special case when $N=n$, we have the following necessary and sufficient conditions on Koopman linearizability to a controllable linear system. These conditions follow directly from the Theorems \ref{th:Neccesary} and \ref{th:Linear_lower} and their proofs, and thus are presented as a corollary:

\begin{corollary}[Controllable Koopman linearization]\label{cor:Controllable} Given any nonlinear system \eqref{eq:nonlinear} with $f(x_0)=0$, there exists a coordinate transformation $\Phi$ that transforms the system into a controllable linear system $(A,B)$ on some set $\tilde{V}$ containing $x_0$ if and only if the two conditions (i) and (ii) of Theorem \ref{th:Linear_lower} hold on $\tilde{V}$, with the constant dimension in condition (i) equal to $n$.
\begin{enumerate}
    \item[(i)] $\operatorname{dim}\left(\operatorname{span}\left\{\operatorname{ad}_f^j g_i(x), 1\le i \le m,\, 0\le j \le n-1 \right\}\right)$ is constant and equal to $n$ for all $x \in \tilde{V}$, and
    \item[(ii)] $\left[\operatorname{ad}_f^k g_i, \operatorname{ad}_f^l g_j\right](x) = 0,$ for all $1\le i, j \le m$ and $k, l \ge 0, \quad \forall x \in \tilde{V}$.
\end{enumerate}
\end{corollary}
\begin{proof} 
    ($\impliedby$) We first notice that the condition (i) above implies that condition (i) of Theorem \ref{th:Linear_lower} holds, because having additional vectors $\operatorname{ad}_f^j g_i(x)$ with $j\ge n$ in the distribution does not increase the dimension beyond $n$. Also, the condition (ii) above is the same as condition (ii) of Theorem \ref{th:Linear_lower}. Thus, our proof follows steps similar to the proof of Theorem \ref{th:Linear_lower}. Specifically, since $r=n$, the transformed state vector $z$ is entirely equal to $z_a$, without any $z_b$ elements. Consequently, equation \eqref{eq:z_dyn} will not have any $\tilde{f}_a,\tilde{f}_b$ terms, and the functions $\tilde{g}_{ai}$ are constants, and the Koopman linear transformation holds over the entire set $\tilde{V}$ instead of on the invariant set $I_x.$ We skip rewriting the detailed steps due to space constraints.\\ 
    
    \noindent ($\implies$) This part of the proof follows from the proof of Theorem \ref{th:Neccesary}. Firstly, as shown in the proof of Theorem \ref{th:Neccesary}, $A^{k}b_i = \nabla_x \Phi(x) \cdot\operatorname{ad}_f^k g_i(x)$ for $k\ge 0$ and $i=1,\ldots,m$. If the Koopman linearized system $(A,B)$ obtained from a transformation $\Phi$ is controllable, the rank of the controllability matrix $[B,\,AB, \cdots, \, A^{n-1}B]$ is equal to $n.$ Note that, due to the Cayley-Hamilton theorem, the power of $A$ beyond $n-1$ does not change the rank. In other words, $\operatorname{span}\left\{A^{k}b_i\right\} = \operatorname{span}\left\{\nabla_x \Phi(x) \cdot\operatorname{ad}_f^k g_i(x)\right\} = \nabla_x \Phi(x) \cdot \operatorname{span}\left\{\operatorname{ad}_f^k g_i(x)\right\}$ with $k\ge 0$, and $i=1,\ldots,m$, is $n-$dimensional. Now, since $\Phi$ is a diffeomorphism, $\nabla_x\Phi$ is non-singular and thus, $ \operatorname{dim}\left(\operatorname{span}\left\{\operatorname{ad}_f^k g_i(x)\right\}\right) = \operatorname{dim}\left(\nabla_x \Phi(x) \cdot \operatorname{span}\left\{\operatorname{ad}_f^k g_i(x)\right\}\right)$ with $k\ge 0$ and $i=1,\ldots,m$, is equal to $n$. This completes our proof.\qed
\end{proof}

The relation between Controllable Koopman linearizability and feedback linearizability is straightforward. Clearly, the former is a special case of feedback linearization, wherein one simply applies identity transformation to the control input. In fact, for single-input systems, the necessary and sufficient conditions for feedback linearization are trivially satisfied if the conditions of Corollary \ref{cor:Controllable} hold. Indeed, for $m=1$, the two conditions of Corollary \ref{cor:Controllable} are:
\begin{eqnarray}
    \operatorname{dim}\left(\operatorname{span}\left\{ \operatorname{ad}_f^0 g_1(x), \ldots, \operatorname{ad}_f^{n-1} g_1(x)\right\}\right)= n-1, \label{eq:dim}\\
    \text{and }\left[\operatorname{ad}_f^k g_1, \operatorname{ad}_f^l g_1\right](x) = 0 \quad \forall k, l \ge 0, \; \forall x \in \tilde{V}. \label{eq:comm}
\end{eqnarray}
Equation \eqref{eq:dim} is the same as condition (i) of Proposition \ref{prop:FL}. Condition (ii) of Proposition \ref{prop:FL} is also satisfied due to equation \eqref{eq:comm}, because for any two elements $v_1(x),v_2(x) \in \Delta(x) \doteq \operatorname{span}\left\{g_1(x),\,\operatorname{ad}_f g_1(x),\, \ldots, \operatorname{ad}_f^{n-2} g_1(x) \right\}$, we have $[v1,v2](x) = 0 \in \Delta(x).$

\section{Examples}\label{sec:results}
\begin{example}\label{ex:slow_manifold1} Consider the following system:
    $$\left[\begin{matrix}\dot{x}_{1}\\ \dot{x}_2\end{matrix}\right] = \left[\begin{matrix}x_{1}\\- x_{1}^{2} + x_{2}\end{matrix}\right] + \left[\begin{matrix}0\\1\end{matrix}\right]u.$$

This system is known to be Koopman linearizable using the transformation $z = \Phi\left(x\right) =  [x_1,x_2,x_1^2]^\top$, leading to the linear system
$$ \dot{z} = \left[\begin{matrix}1 & 0 & 0\\0 & 1 & -1\\0 & 0 & 2\end{matrix}\right]z + \left[\begin{matrix}0 \\1 \\ 0\end{matrix}\right]u.$$ This means that the necessary conditions of Theorem \ref{th:Neccesary} must hold. Indeed, for this system, we have $$\left\{g, \operatorname{ad}_f^1(g), \operatorname{ad}_f^2(g), \ldots\right\} = \left\{\left[ \begin{array}{c}
     0  \\
     1 
\end{array}\right], \left[ \begin{array}{c}
     0  \\
     1 
\end{array}\right], \left[ \begin{array}{c}
     0  \\
     0 
\end{array}\right],\ldots \right\},$$ which is of constant rank 1 everywhere, and $\left[\operatorname{ad}_f^k g, \operatorname{ad}_f^l g\right](x)=0, \quad$ for all $ k, l \ge 0$ everywhere on $\mathbbm{R}^2$. Conversely, since these two conditions hold, we can apply Theorem \ref{th:Linear_lower} which guarantees the existence of an invariant manifold $I_x$ and a corresponding linearizing transformation defined on the invariant manifold. Specifically for this example, $I_x = \left\{x \mid x_1 = 0\right\}$ and $z(t) = \left[\begin{array}{c}
     x_2(t)  \\
     x_1(t) 
\end{array}\right]$ evolves according to a linear dynamics of the form
\[
    \dot{z} = \left[\begin{array}{cc}
        1 & * \\
        0 & *
    \end{array}\right]z + \left[\begin{array}{c}
         1  \\
         0 
    \end{array}\right]u, 
\]
on the invariant manifold $I_z = \{z \mid z_2 = 0\}.$ The matrix elements indicated by $*$ do not influence the dynamics on the manifold $I_z$, and can be taken to be any number.
\end{example}
 
\begin{example}\label{ex:slow_manifold2}
    We next consider a system very similar to the previous example, but a flipped control input:
     $$\left[\begin{matrix}\dot{x}_{1}\\ \dot{x}_2\end{matrix}\right] = \left[\begin{matrix}x_{1}\\- x_{1}^{2} + x_{2}\end{matrix}\right] + \left[\begin{matrix}1\\0\end{matrix}\right]u.$$ Applying Theorem \ref{th:Neccesary}, one can conclude that this system is not Koopman linearizable around any neighborhood of the origin. This is because 
     \begin{gather*}
         \left\{g, \operatorname{ad}_f^1(g), \operatorname{ad}_f^2(g), \ldots\right\} \\ =
         \left\{\left[ \begin{array}{c}
     1  \\
     0 
\end{array}\right], \left[ \begin{array}{c}
     1  \\
     -2x_1 
\end{array}\right], \left[ \begin{array}{c}
     -1  \\
     2x_1 
\end{array}\right],\ldots \right\},
 \end{gather*}
 has a rank of $1$ in the set $\left\{(x_1,x_2)\mid x_1=0\right\}$ that contains the origin, but has rank $2$ everywhere else. This violates the condition (i) in Theorem \ref{th:Neccesary}. In fact, the condition (ii) is also violated since $[\operatorname{ad}_f^0(g),\operatorname{ad}_f^1(g)] = \left[\begin{array}{c}
     0  \\
     -2 
\end{array}\right]\neq \mathbf{0}.$ One can also verify by first principles that this system cannot be Koopman linearized, by noticing that for any candidate transformation $z = \Phi(x)$, we firstly need $\nabla_x \Phi\cdot \left[ \begin{array}{c}
     1  \\
     0 
\end{array}\right] = \frac{\partial}{\partial x_1} \Phi(x)$ to be a constant vector (say, $\gamma$). This means two things. First, $x_1$ appears as a linear term in $\Phi(x)$ and second, $\frac{\partial^2}{\partial x_2 x_1} \Phi(x) = 0$, i.e., $\frac{\partial}{\partial x_2} \Phi(x)$ is a vector independent of $x_1$, which we denote as $\lambda(x_2)$. This $\lambda(x_2)$ must be nonzero; otherwise, $\Phi(x)$ would not be a diffeomorphism since it would have to be independent of $x_2$. Thus, with $u \equiv 0$,
\begin{eqnarray*}
    \dot{\Phi}(x) &=& \frac{\partial \Phi}{\partial x_1} \dot{x}_1 + \frac{\partial \Phi}{\partial x_2} \dot{x}_2 \\
    &=& \gamma x_1 -  \lambda(x_2) x_1^2 +  \lambda(x_2) x_2\\
    &\neq& A \Phi(x) \text{ for any square matrix } A,
\end{eqnarray*}
because the vector $\Phi(x)$ does not have any element with terms involving $x_1^2$, as we argued earlier that it can only linearly depend on $x_1.$ 
\end{example}

\begin{remark}
    Both these two examples are not feedback linearizable to a controllable system, but by lifting to a higher-dimensional state space, we can Koopman linearize Example \ref{ex:slow_manifold1}. 
\end{remark}

\begin{example}\label{ex:FB_Not_Koop}
The system given by
$$\left[\begin{matrix}\dot{x}_{1}\\ \dot{x}_2\end{matrix}\right] = \left[\begin{matrix}x_{1}^2 + x_2\\0\end{matrix}\right] + \left[\begin{matrix}0\\1\end{matrix}\right]u,$$ is feedback linearizable to a controllable system, if we apply the state transformation $z = \Phi(x) = [x_1, x_1^2 + x_2]$ and input tranformation $v = u + 2x_1(x_1^2 + x_2)$. The feedback linearized system is a double integrator $\dot{z}_1 = z_2$, $\dot{z}_2 = v.$
     \begin{gather*}
         \left\{g, \operatorname{ad}_f^1(g), \operatorname{ad}_f^2(g), \ldots\right\} \\ =
         \left\{\left[ \begin{array}{c}
     0  \\
     1 
\end{array}\right], \left[ \begin{array}{c}
     1  \\
     0 
\end{array}\right], \left[ \begin{array}{c}
     2x_1  \\
     0 
\end{array}\right],\ldots \right\},
 \end{gather*}
 has a rank of $2$ everywhere on $\mathbbm{R}^2.$ Thus, condition (i) in Theorem \ref{th:Neccesary} is satisfied. However, the condition (ii) is violated because $[\operatorname{ad}_f^2(g),\operatorname{ad}_f^1(g)] = \left[\begin{array}{c}
     2  \\
     0 
\end{array}\right]\neq \mathbf{0}.$  Thus, we can conclude from \ref{th:Neccesary} that this system is not Koopman linearizable. Just like in Example \ref{ex:slow_manifold2}, one can verify this claim independently, by proceeding with the assumption that a diffeomorphic transformation $\Phi(x)$ leads to a Koopman linear dynamics $\dot\Phi(x) = \frac{\partial \Phi}{\partial x_1}\cdot (x_1^2+x_2) +  \frac{\partial \Phi}{\partial x_2}\cdot u = A\Phi(x) + Bu$ for some matrices $A$ and $B$, and then arriving at a contradiction that $\Phi$ depends only on $x_2$.
\end{example}

\begin{example}
Consider the system
    \begin{equation*}
        \left[\begin{matrix}\dot{x}_{1}\\ \dot{x}_2\end{matrix}\right] =\left[\begin{matrix}1 - e^{- x_{1}}\\2 x_{1} - 2 x_{1} e^{- x_{1}}\end{matrix}\right] + \left[\begin{matrix}- e^{- x_{1}} & 0\\- 2 x_{1} e^{- x_{1}} - 1 & -1\end{matrix}\right]\left[\begin{matrix}u_1\\ u_2 \end{matrix}\right]
    \end{equation*}

    The diffeomorphic transformation $z = \Phi(x) = \left[\begin{array}{c}
     x_1^2 - x_2  \\
     1 - e^{x_1} 
\end{array}\right]$ evolves according to a linear dynamics of the form
\[
    \dot{z} = \left[\begin{array}{cc}
        0 & 0 \\
        0 & 1
    \end{array}\right]z + \left[\begin{array}{cc}
        1 & 1 \\
        1 & 0
    \end{array}\right]u, 
\]
which can be verified to be a controllable linear system, with the rank of the controllability matrix $[B, AB]$ equal to 2. Corollary \ref{cor:Controllable} provides necessary and sufficient conditions for the existence of such a transformation, which hold true for this system:  \\
\noindent\textbf{(i) Rank condition of Corollary \ref{cor:Controllable}}:
For $k=0$, we have $\operatorname{ad}_f^k(g_1)=g_1(x)=\left[ \begin{array}{c}
     - e^{- x_{1}}  \\
     - 2 x_{1} e^{- x_{1}} - 1 
\end{array}\right]$ and $\operatorname{ad}_f^k(g_2)=g_2(x)=\left[ \begin{array}{c}
     0  \\
     -1 
\end{array}\right].$ For $k>0$, we have
$$
\operatorname{ad}_f^k(g_1) = (-1)^{k+1}\left[ \begin{array}{c}
     e^{- x_{1}}  \\
     2 x_{1} e^{- x_{1}} 
\end{array}\right] \text{ and }
\operatorname{ad}_f^k(g_2) = \left[ \begin{array}{c}
     0  \\
     0 
\end{array}\right].
$$
Clearly, $\operatorname{dim}\left(\operatorname{span}\left\{\operatorname{ad}_f^j g_i(x), 1\le i \le m,\, j\ge 0\right\}\right) = 2$ for all $x\in\mathbbm{R}^2$ since $g_1(x)$ and $g_2(x)$ are linearly independent. 

\noindent\textbf{(ii) Commuting bracket condition of Corollary \ref{cor:Controllable}}: Notice that for all $k\ge 0$, we have $\nabla_x \operatorname{ad}_f^k g_1 = \left[\begin{array}{cc}
        * & 0 \\
        * & 0
    \end{array}\right]$ and thus $\left[\operatorname{ad}_f^k g_1, \operatorname{ad}_f^l g_2\right](x) = \nabla_x \operatorname{ad}_f^k g_1 \cdot \operatorname{ad}_f^l g_2 = \left[\begin{array}{cc}
        * & 0 \\
        * & 0
    \end{array}\right]\left[ \begin{array}{c}
     0  \\
     * 
\end{array}\right] = 0$ for $k,l \ge 0.$ Next, $\left[\operatorname{ad}_f^k g_2, \operatorname{ad}_f^l g_2\right](x) = 0$ for $k,l \ge 0$ since Lie bracket between constant vectors is zero. Finally, for $k,l > 0$, we have $\left[\operatorname{ad}_f^k g_1, \operatorname{ad}_f^l g_1\right](x) = 0$ since both vectors inside the Lie bracket are the same (up to a constant scalar multiplicative factor of $\pm 1$). When $k > 0$ and $l=0$, we have $\left[\operatorname{ad}_f^k g_1, \operatorname{ad}_f^l g_1\right](x) = \left[(-1)^k (g_1 - g_2), g_1\right](x) = (-1)^k\left[ g_1 , g_1\right](x) - (-1)^k\left[ g_2 , g_1\right](x) = 0 - 0 = 0.$ The case for $l > 0$ and $k=0$ is the same. Thus, $\left[\operatorname{ad}_f^k g_i, \operatorname{ad}_f^l g_j\right](x) = 0,$ for all $1\le i, j \le 2,$ and $ k, l \ge 0, \quad \forall x \in \mathbbm{R}^2$.

\end{example}

\section{Discussions and conclusion}\label{sec:conc}
Research on Koopman linear representation of nonlinear control systems has undergone rapid progress. Many important works have been conducted on learning approximate Koopman linearization and corresponding bounds on the residual errors. Serving as a precursor to Koopman learning, this work studies the conditions under which one can perform an exact Koopman linearization of control-affine nonlinear systems. The necessary conditions for Koopman linearization (presented in Theorem \ref{th:Neccesary}), are also sufficient for Koopman linearization restricted to a control-invariant set (Theorem \ref{th:Linear_lower}). The sufficient condition thus gives us a slightly weaker notion of Koopman linearization, and our future work will focus on tightening the sufficient condition. Transformation to controllable linear systems are studied next, with an additional rank condition which, along with the previous conditions, is necessary and sufficient for controllable Koopman linearization.

Considering that the differential geometric tools used for studying Koopman linearization in this paper produce conditions that, at a cursory glance, appear similar to feedback linearization conditions, we present Figure \ref{fig:Venn} to demarcate between the two types of linearization. In particular, Koopman linearizability does not necessarily imply feedback linearizability (Example \ref{ex:slow_manifold1}), and conversely, a feedback linearizability of a nonlinear control-affine system does not necessarily imply Koopman linearizability (Example \ref{ex:FB_Not_Koop}). An open question that remains to be answered is whether there exists a class of systems, which are feedback linearizable as well as can be linearized without input transformation purely via lifting.

Finally, we would like to highlight that our results impose fundamental restrictions on Koopman representation learning algorithms. However, since Koopman linearization is increasingly being adopted towards surrogate modeling of unknown systems via data-driven algorithms, being able to determine Koopman linearizability without explicit knowledge of the drift and control vectors (i.e. $f(x)$ and $g_i(x)$) is an important practical consideration. In our future work, we will explore the geometric characterization of Lie brackets in terms of an infinitesimal commutator (Theorem 1.33 in \cite{olver1993applications}) to determine Koopman linearizability of any given control-affine system directly from trajectory data.

\begin{figure}[t]
    \centering
\includegraphics[width=0.7\columnwidth,trim=0 15 0 5, clip]{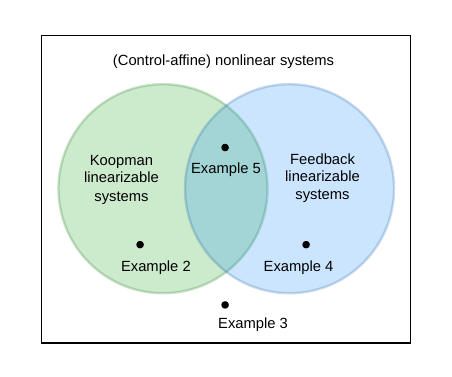}
    \caption{Relationship between Koopman linearizable systems and feedback linearizable systems.}
    \label{fig:Venn}
\end{figure}


\section*{References}
\bibliographystyle{ieeetr}
\bibliography{Ref,Feedback_linearization_refs,mendeley}

\end{document}